\documentclass[11pt,a4paper]{article}
\usepackage{mathrsfs}
\usepackage{amsmath,amsthm,color,eepic}
\usepackage{amssymb} \usepackage{verbatim}

\newtheorem{thm}{Theorem}[section]

\newtheorem{dfn}[thm]{Definition}
\newtheorem{claim}[thm]{Claim}
\newtheorem{lemma}[thm]{Lemma}

\newtheorem{prop}[thm]{Proposition}

\newtheorem{corollary}[thm]{Corollary}
\newtheorem{conj}[thm]{Conjecture}

\begin{document}

\title{Extremal problems of Erd\H{o}s, Faudree, Schelp and Simonovits on paths and cycles}

\date{}

\author{Binlong Li\thanks{$^a$School of Mathematics and Statistics, Northwestern Polytechnical University,
Xi'an, Shaanxi 710072, China. $^b$Xi'an-Budapest Joint Research Center for Combinatorics,
Northwestern Polytechnical University, Xi'an, Shaanxi, 710129, China.} \and Jie
Ma\thanks{School of Mathematical Sciences, University of Science and
Technology of China, Hefei, Anhui, 230026, China.} \and Bo Ning\thanks{College of Computer Science, Nankai
University, Tianjin, 300071, China. Email: bo.ning@nankai.edu.cn.} }

\maketitle

\begin{center}
\begin{abstract}
For positive integers $n>d\geq k$, let $\phi(n,d,k)$ denote the least
integer $\phi$ such that every $n$-vertex graph with at least $\phi$
vertices of degree at least $d$ contains a path on $k+1$ vertices.
Many years ago, Erd\H{o}s, Faudree, Schelp and Simonovits proposed
the study of the function $\phi(n,d,k)$,
and conjectured that for any positive integers $n>d\geq k$, it holds that
$\phi(n,d,k)\leq \lfloor\frac{k-1}{2}\rfloor\lfloor\frac{n}{d+1}\rfloor+\epsilon$,
where $\epsilon=1$ if $k$ is odd and $\epsilon=2$ otherwise.
In this paper we determine the values of the function $\phi(n,d,k)$ exactly.
This confirms the above conjecture of Erd\H{o}s et al. for all positive integers $k\neq 4$
and in a corrected form for the case $k=4$.
Our proof utilizes, among others, a lemma of Erd\H{o}s et al. \cite{EFSS89}, a theorem of Jackson \cite{J81},
and a (slight) extension of a very recent theorem of Kostochka, Luo and Zirlin \cite{KLZ},
where the latter two results concern maximum cycles in bipartite graphs.
Moreover, we construct examples to provide answers to two closely related questions raised by Erd\H{o}s et al.
\end{abstract}


\end{center}

\section{Introduction}
We consider the following extremal problem asked by Erd\H{o}s,
Faudree, Schelp and Simonovits in \cite{EFSS89}: for given
positive integers $n>d\geq k$, what is the minimum value $\ell$
such that every $n$-vertex graph with at least $\ell$ vertices
of degree at least $d$ contains a path $P_{k+1}$ on $k+1$ vertices?
The goal of the present paper is to provide a complete solution
for all positive integers $n>d\geq k$.

One of the best known results in extremal graph theory is the
Erd\H{o}s-Gallai Theorem \cite{EG59}, which states that any
$n$-vertex graph with more than $(k-1)n/2$ edges contains a
path on $k+1$ vertices. Since then, there have been many other
extremal results on the existence of long paths in graphs with
a large number of edges or vertices of high-degree. In this paper
we investigate the following function, the study of what was proposed
by Erd\H{o}s, Faudree, Schelp and Simonovits \cite{EFSS89}.

\begin{dfn}\label{Def:phi}
For positive integers $n>d\geq k$,
define $\phi(n,d,k)$ to be the smallest integer $\phi$ such that
every $n$-vertex graph with at least $\phi$ vertices of degree at
least $d$ contains a path $P_{k+1}$ on $k+1$ vertices.
\end{dfn}

In this language, the well-known theorem of Dirac \cite{D52} asserts that $\phi(n,d,d)\leq n$
and by that, we see that the function $\phi(n,d,k)$ is well-defined if and only if $d\geq k$.
A result of Bazgan, Li and Wo\'{z}niak \cite{BLW00} shows that $\phi(n,d,d)\leq \frac{n}{2}$.
For the general case, Erd\H{o}s, Faudree, Schelp and Simonovits \cite{EFSS89} announced that for any $k$,
there exists a constant $c$ such that if $n$ is large enough with respect to $k$,
then $\phi(n,d,k)\leq \lfloor \frac{k-1}{2}\rfloor \lfloor \frac{n}{d+1}\rfloor+c.$
They \cite{EFSS89} further made the following conjecture (also see \cite{FRS97}).

\begin{conj}[Erd\H{o}s, Faudree, Schelp and Simonovits, \cite{EFSS89}]\label{Conj:EFSS}
Let $n,d$ and $k$ be any positive integers with $n>d\geq k$.
Then $\phi(n,d,k)\leq \lfloor \frac{k-1}{2}\rfloor\lfloor\frac{n}{d+1}\rfloor+\epsilon$,
where $\epsilon=1$ if $k$ is odd and $\epsilon=2$ otherwise.
\end{conj}

Much attention in \cite{EFSS89} was devoted to the special case when $d+1\leq n\leq 2d+1$.
In this case, the authors \cite{EFSS89} showed that approximately $k/2$ vertices of
degree at least $d$ are enough to ensure the existence of $P_{k+1}$.
They also commented that ``unfortunately, even for this interval of values
we are not able to prove the exact statement of the conjecture."
However, as we shall see later, this case (i.e., $d+1\leq n\leq 2d+1$)
is a major difficulty that we face and where several novel ideas
take place in our proof.

Our main result determines the function $\phi(n,d,k)$ completely in the following statement.

\begin{thm}\label{Thm:Main}
For any positive integers $n,d$ and $k$ with $n>d\geq k$, the followings are true:\\
(i) If $k$ is odd, then
$\phi(n,d,k)=\frac{k-1}{2}q+1$, where $n=q(d+1)+r$ with $0\leq r\leq d$.\\
(ii) If $k$ is even, then

(a) for $k=2$, $\phi(n,d,2)=1$;

(b) for $k=4$, $\phi(n,d,4)=\left\{\begin{array}{ll}
  2q+1, & 0\leq r\leq d;\\
  2q+2, & d<r<2d,
\end{array}\right.$ where $n=2qd+r$ with $0\leq r<2d$;

(c) for $k\geq 6$, $\phi(n,d,k)=\left\{\begin{array}{ll}
  \frac{k-2}{2}q+1, & 0\leq r\leq d-\frac{k}{2};\\
  \frac{k-2}{2}q+2, & d-\frac{k}{2}<r\leq d,
\end{array}\right.$ where $n=q(d+1)+r$ with $0\leq r\leq d$.
\end{thm}

We see immediately that
$\phi(n,d,k)=\lfloor\frac{k-1}{2}\rfloor \lfloor\frac{n}{d+1}\rfloor+1$ when $k$ is odd,
and
$\phi(n,d,k)\leq\lfloor\frac{k-1}{2}\rfloor
\lfloor\frac{n}{d+1}\rfloor+2$ when $k\neq 4$ is even.
However, the case $k=4$ is different.
In summary, we have the following.
\begin{corollary}
Conjecture~\ref{Conj:EFSS} is true for any integer $k\neq 4$ and false for $k=4$.
\end{corollary}

Our proof of Theorem~\ref{Thm:Main} is inductive in its nature.
For that it is crucial for us to manage the base case when $d+1\leq n\leq 2d+1$.
It turns out in the proof of the base case that we make use of results
on maximum cycles in bipartite graphs due to Jackson \cite{J81}
and Kostochka, Luo and Zirlin \cite{KLZ} (see Theorems~\ref{Thm:Jackson}
and \ref{Thm:KLZ}, respectively). To be precise, we partition
the vertex set into two parts $X$ and $Y$, where $X$ consists of
vertices of degree at least $d$, and then apply the above results
on maximum cycles to find a small number of disjoint paths in the
bipartite subgraph $G(X,Y)$ to cover all vertices in $X$;
finally, an application of a lemma of Erd\H{o}s et al.
in \cite{EFSS89} (see Lemma~\ref{Thm:EFSS-Path}) will ensure
the desired long path. We would like to point out that it
seems to be an incredible coincidence that the bounds we need
in this argument are exactly what the recent result of Kostochka,
Luo and Zirlin \cite{KLZ} provided. The case when $n\geq 2d+2$
will be handled differently, which is essentially reduced to the base case.

The rest of this paper is organized as follows.
In Section~\ref{sec:ext-graph}, we construct extremal graphs for
the function $\phi(n,d,k)$ and establish the lower bound of Theorem~\ref{Thm:Main}.
In Section~\ref{Sec:Preliminaries}, we introduce the notation, a lemma of Erd\H{o}s et al.,
and results of Jackson and Kostochka, Luo and Zirlin on maximum cycles in bipartite graphs;
we also provide some variants and extension of these cycle results for the coming proof.
In Section~\ref{Sec:Thm:Pr-Main}, we complete the proof of Theorem \ref{Thm:Main}.
In Section \ref{Sec:ErdosProb}, we give better constructions to answer
two questions in \cite{EFSS89} which are closely related to Conjecture~\ref{Conj:EFSS}.

\section{Extremal graphs}\label{sec:ext-graph}

In this section, we construct extremal graphs for the function $\phi(n,d,k)$.
This will give the matched lower bound of $\phi(n,d,k)$ in Theorem~\ref{Thm:Main}.

We start with some notation. Let $n> d\geq k$ be positive integers and
let $G, H$ be two graphs. Throughout this paper, the word {\it disjoint}
always means for {\it vertex-disjoint} unless otherwise specified.
By $G+H$ we mean the disjoint union of $G$ and $H$, and we use $k\cdot G$
to denote the union of $k$ disjoint copies of the same graph $G$.
Let $K_n$ be the $n$-vertex clique, $I_n$ be the graph induced by an
independent set of $n$ vertices, and $K_{1,n}$ be the star with $n$ leaves.
We define two special yet important graphs as following (see Figure~\ref{Fig:Hdk}):
\begin{itemize}
\item The graph $H_{d,k}$ is obtained from the disjoint union of
$K_{\lfloor\frac{k-1}{2}\rfloor}$ and $I_{d+1-\lfloor\frac{k-1}{2}\rfloor}$
by joining every vertex of $K_{\lfloor\frac{k-1}{2}\rfloor}$ to every vertex
of $I_{d+1-\lfloor\frac{k-1}{2}\rfloor}$.
\item For even integers $k\geq 4$, let $H_{d,k}^*$ be the graph obtained from
$H_{d,k}$ by adding a disjoint copy of $I_{d+1-\frac{k}{2}}$ and joining every
vertex in $I_{d+1-\frac{k}{2}}$ to a fixed vertex of degree $\frac{k}{2}-1$ in $H_{d,k}$.
\end{itemize}
Note that $H_{d,k}$ has $d+1$ vertices in total and $\lfloor\frac{k-1}{2}\rfloor$ vertices of degree at least $d$,
while $H^*_{d,k}$ has $2d+2-\frac{k}2$ vertices in total and $\frac{k}{2}$ vertices
of degree at least $d$. In particular, $H^*_{d,4}$ is the graph obtained from
two disjoint stars on $d$ vertices by joining the two centers (we will also call
it a {\it double-star}).

With the above notation, now we define the extremal graph $G(n,d,k)$ for the function $\phi(n,d,k)$.

\begin{dfn}\label{Def:Hdk}
For positive integers $n> d\geq k$,
we define the graph $G(n,d,k)$ as follows.

\begin{itemize}
\item For $k\in \{1,2\}$, let $G(n,d,k)=I_n$.

\item For $k=4$, write $n=2qd+r$ with $0\leq r<2d$ and let
$$G(n,d,k)=\left\{\begin{array}{ll} q\cdot H^*_{d,4}+ I_r,  & \mbox{if } r\leq d;\\
q\cdot H^*_{d,4}+K_{1,d}+I_{r-d-1},  & \mbox{otherwise.}
\end{array}\right.$$

\item For odd $k\geq 3$, write $n=q(d+1)+r$ with $0\leq r\leq d$ and let $G(n,d,k)=q\cdot H_{d,k}+I_r$.

\item For even $k\geq 6$, write $n=q(d+1)+r$ with $0\leq r\leq d$
and let $$G(n,d,k)=\left\{\begin{array}{ll} q\cdot H_{d,k}+I_r,  & \mbox{if } r\leq d-\frac{k}{2};\\
(q-1)\cdot H_{d,k}+H_{d,k}^*+I_{r-d+\frac{k}{2}-1},  & \mbox{otherwise.}
\end{array}\right.$$
\end{itemize}
\end{dfn}

It is straightforward to check the following fact on $G(n,d,k)$.

\begin{lemma}\label{Lem:Ext-graphs}
For any positive integers $n>d\geq k$,
the $n$-vertex graph $G(n,d,k)$ contains no $P_{k+1}$ and thus the lower bound of Theorem~\ref{Thm:Main} holds.
\end{lemma}

\begin{figure}[h]
    \begin{center}
        \begin{picture}(350,150)

        \newcommand{\tuoyuan}[2]{\qbezier(#1,0)(#1,#2)(0,#2)
            \qbezier(0,#2)(-#1,#2)(-#1,0) \qbezier(-#1,0)(-#1,-#2)(0,-#2)
            \qbezier(0,-#2)(#1,-#2)(#1,0)}

        \put(140,0){\put(-120,75){\thinlines\tuoyuan{10}{45}}
            \put(-130,75){\line(1,1){20}}
            \put(-130,65){\line(1,1){20}}
            \put(-130,55){\line(1,1){20}}
            \put(-130,85){\line(1,1){18}}
            \put(-129,95){\line(1,1){16}}
            \put(-128,105){\line(1,1){12}}
            \put(-128,45){\line(1,1){17}}
            \put(-126,35){\line(1,1){16}}

            \multiput(-60,30)(0,15){3}{\circle*{4}}
            \multiput(-60,90)(0,15){3}{\circle*{4}}
            \multiput(-60,70)(0,5){3}{\circle*{1}}
            \multiput(-120,35)(0,20){2}{\circle*{4}}
            \multiput(-120,65)(0,10){3}{\circle*{1}}
            \multiput(-120,95)(0,20){2}{\circle*{4}} \put(-60,30){\line(-12,1){60}}
            \put(-60,30){\line(-12,5){60}} \put(-60,30){\line(-12,13){60}}
            \put(-60,30){\line(-12,17){60}} \put(-60,45){\line(-6,-1){60}}
            \put(-60,45){\line(-6,1){60}} \put(-60,45){\line(-6,5){60}}
            \put(-60,45){\line(-6,7){60}} \put(-60,60){\line(-12,-5){60}}
            \put(-60,60){\line(-12,-1){60}} \put(-60,60){\line(-12,7){60}}
            \put(-60,60){\line(-12,11){60}}
            \put(-60,90){\line(-12,-11){60}}
            \put(-60,90){\line(-12,-7){60}} \put(-60,90){\line(-12,1){60}}
            \put(-60,90){\line(-12,5){60}} \put(-60,105){\line(-6,-7){60}}
            \put(-60,105){\line(-6,-5){60}} \put(-60,105){\line(-6,-1){60}}
            \put(-60,105){\line(-6,1){60}} \put(-60,120){\line(-12,-17){60}}
            \put(-60,120){\line(-12,-13){60}} \put(-60,120){\line(-12,-5){60}}
            \put(-60,120){\line(-12,-1){60}} \put(-170,72){$K_{\lfloor \frac{k-1}{2} \rfloor}$}
            \put(-90,0){$H_{d,k}$}
            \put(-65,72){$\left.
                \begin{aligned}
                ~ \\
                ~ \\
                ~ \\
                ~ \\

                ~
                \end{aligned}
                \right\} d+1-\lfloor \frac{k-1}{2} \rfloor$}
        }

    ~~~~~~~~~~~~~

        \put(200,0){\put(-10,75){\thinlines\tuoyuan{10}{45}}
            \multiput(50,15)(0,15){3}{\circle*{4}}
            \multiput(50,105)(0,15){3}{\circle*{4}}
            \put(50,75){\circle*{4}}
            \multiput(50,85)(0,5){3}{\circle*{1}}
            \multiput(50,55)(0,5){3}{\circle*{1}}
            \multiput(-10,35)(0,20){2}{\circle*{4}}
            \multiput(-10,65)(0,10){3}{\circle*{1}}
            \multiput(-10,95)(0,20){2}{\circle*{4}}
            \multiput(90,15)(0,15){3}{\circle*{4}}
            \multiput(90,105)(0,15){2}{\circle*{4}}
            \multiput(90,61)(0,13){3}{\circle*{1}}

            \put(-20,75){\line(1,1){20}}
            \put(-20,65){\line(1,1){20}}
            \put(-20,55){\line(1,1){20}}
            \put(-20,85){\line(1,1){18}}
            \put(-19,95){\line(1,1){16}}
            \put(-18,105){\line(1,1){12}}
            \put(-18,45){\line(1,1){17}}
            \put(-16,35){\line(1,1){16}}

            \put(50,30){\line(-12,1){60}} \put(50,30){\line(-12,5){60}}
            \put(50,30){\line(-12,13){60}} \put(50,30){\line(-12,17){60}}
            \put(50,45){\line(-6,-1){60}} \put(50,45){\line(-6,1){60}}
            \put(50,45){\line(-6,5){60}} \put(50,45){\line(-6,7){60}}
            \put(50,15){\line(-12,8){60}} \put(50,15){\line(-12,4){60}}
            \put(50,15){\line(-12,16){60}} \put(50,15){\line(-12,20){60}}
            \put(50,75){\line(-3,-2){60}} \put(50,75){\line(-3,-1){60}}
            \put(50,75){\line(-3,1){60}} \put(50,75){\line(-3,2){60}}

            \put(50,135){\line(-12,-16){60}} \put(50,135){\line(-12,-20){60}}
            \put(50,135){\line(-12,-8){60}} \put(50,135){\line(-12,-4){60}}
            \put(50,105){\line(-6,-7){60}} \put(50,105){\line(-6,-5){60}}
            \put(50,105){\line(-6,-1){60}} \put(50,105){\line(-6,1){60}}
            \put(50,120){\line(-12,-17){60}} \put(50,120){\line(-12,-13){60}}
            \put(50,120){\line(-12,-5){60}} \put(50,120){\line(-12,-1){60}}

            \put(50,75){\line(12,-18){40}} \put(50,75){\line(12,-9){40}}
            \put(50,75){\line(12,-14){40}}
            \put(50,75){\line(12,9){40}} \put(50,75){\line(12,14){40}}
            \put(-55,72){$K_{\frac{k}{2}-1}$} \put(55,0){$H_{d,k}^*$}
            \put(85,63){$\left.
                \begin{aligned}
                ~ \\
                ~ \\
                ~ \\
                ~ \\
                ~\\
                ~
                \end{aligned}
                \right\} d+1-\frac{k}{2}$}
        }
        \end{picture}
        \caption{$H_{d,k}$ (for all $k$) and $H_{d,k}^*$ (for even $k$)}
        \label{Fig:Hdk}
    \end{center}
\end{figure}
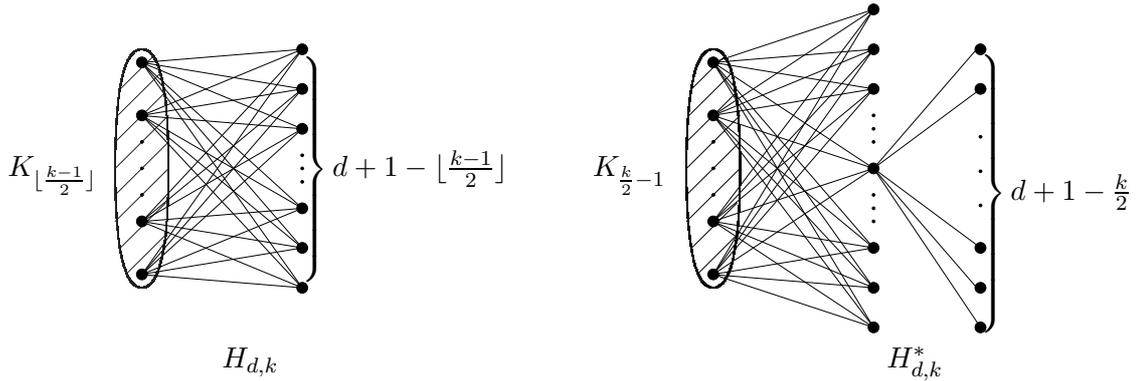

\section{Preliminaries and some results on bipartite graphs}\label{Sec:Preliminaries}
Let $G$ be a graph. For disjoint subsets $X, Y\subseteq V(G)$,
we use $G(X,Y)$ to denote the bipartite subgraph of $G$ induced by two parts $X$ and $Y$.
Let $P$ be a path or a cycle.
By $|P|$, we mean the number of vertices in $P$.
For $x,y\in V(P)$, let $xPy$ be a subpath of $P$ between $x$ and $y$.
When $P$ is associated with an orientation, the successor and predecessor of $x$ along the direction are denoted by
$x^+$ and $x^-$ (if they exist), respectively. We also denote by $x^{++}:=(x^+)^+$ and $x^{--}:=(x^-)^-$.
For a subset $S\subseteq V(G)$, we define $N(S)$ to be the set of all vertices $x\in V(G)\backslash S$ which is adjacent to some vertex in $S$.
If $S$ consists of a single vertex $x$, then we write $N(S)$ as $N(x)$.

We now introduce two theorems on the existence of maximum cycles in bipartite graphs,
which provide crucial tools for the proof of our main result.
The first one is due to Jackson \cite{J81}.

\begin{thm}[\rm {Jackson, \cite[Theorem~1]{J81}}]\label{Thm:Jackson}
Let $G$ be a bipartite graph with two parts $X$ and $Y$.
If $2\leq |X|\leq d$, $|Y|\leq 2d-2$, and every vertex in $X$ has degree at least $d$,
then $G$ has a cycle containing all vertices in $X$.
\end{thm}

The following theorem, conjectured by Jackson \cite{J81} and proved by Kostochka, Luo and Zirlin \cite{KLZ} recently,
strengthens the above theorem of Jackson for 2-connected graphs.

\begin{thm}[\rm {Kostochka et al., \cite[Theorem~1.6]{KLZ}}]\label{Thm:KLZ}
Let $G$ be a 2-connected bipartite graph with two parts $X$ and $Y$.
If $2\leq |X|\leq d$, $|Y|\leq 3d-5$, and every vertex in $X$ has degree at least $d$,
then $G$ has a cycle containing all vertices in $X$.
\end{thm}

Our proof actually needs some intermediate statements from the proof of \cite{KLZ}.
Let us give some notation used in \cite{KLZ} first.
Let $G$ be bipartite with parts $X$ and $Y$ which is not a forest.
For a cycle $C$ and a vertex $x$ in $G$,
we say $(C,x)$ is a {\it tight pair} if $C$ is a longest cycle in $G$, $x\in X\backslash V(C)$,
and subject to these, $d_C(x):=|N(x)\cap V(C)|$ is maximum.
Clearly $G$ has a tight pair if and only if
$G$ has no cycle containing all vertices in $X$.
The followings are collected from the proof of Theorem~1.6 in \cite{KLZ}.

\begin{lemma}[\rm {Kostochka et al., \cite{KLZ}}]\label{lem:KLZ}
Let $G$ be a bipartite graph with two parts $X$ and $Y$ such that $|X|\leq d\leq\min\{d(x): x\in X\}$.
Let $(C,x)$ be a tight pair in $G$ with $c=|C|/2$. Then the followings hold:
\begin{itemize}
\item [(i)] If $d_C(x)\leq 1$ and there is a path connecting two vertices
in $C$ and passing through $x$, then $|Y|\geq 3d-4$ (see Case 1 in the proof of Theorem~1.6 in \cite{KLZ});
\item [(ii)] If $2\leq d_C(x)<c$, then $|Y|\geq 3d-4$ (i.e., Lemma~2.6 in \cite{KLZ});
\item [(iii)] If $d_C(x)=c$, then for each $x_i\in X\cap V(C)$ and each $y\in N_{G-C}(x_i)$, $x_i$ is a cut-vertex separating $y$ from $V(C)-x_i$ (i.e., Lemma~2.7 in \cite{KLZ}).
\end{itemize}
\end{lemma}

Let $G$ be a connected graph which is not a forest.
We say that $G$ is {\it essentially-2-connected}, if $G-V_1$
is 2-connected, where $V_1$ denotes the set of vertices of degree one in $G$.
We need a variance of Theorem~\ref{Thm:KLZ} for essentially-2-connected graphs.

\begin{lemma}\label{Lem:Refine-KLZ}
Let $G$ be an essentially 2-connected bipartite graph with parts $X$ and $Y$.
Suppose that $2\leq |X|\leq d-1$, $|Y|\leq 3d-5$, and every vertex in $X$
has degree at least $d$. Then $G$ has a cycle containing all
vertices in $X$.
\end{lemma}

\begin{proof}
Suppose for a contradiction that $G$ has no cycle containing all vertices in $X$.
Let $(C,x)$ be a tight pair of $G$ and $c=|C|/2$. So $c<|X|$.
Since every vertex in $X$ has degree at least $d\geq 3$,
we see that the vertices in $X\cup V(C)$ are contained in the 2-connected subgraph $G-V_1$
and thus there is a path connecting two vertices in $C$ and passing through $x$.
If $d_C(x)<c$, then by Lemma~\ref{lem:KLZ} (i) and (ii), we get $|Y|\geq 3d-4$, a contradiction.
So we assume that $d_C(x)=c$, i.e., $x$ is adjacent to all vertices in $Y\cap V(C)$.

We claim that every vertex $y\in Y\backslash V(C)$ has degree one in $G$.
Suppose otherwise that there exists some $y\in Y\backslash V(C)$ with $d(y)\geq 2$.
Since $G$ is essentially 2-connected,
there is a path $P$ connecting two vertices in $C$, passing through $y$, and internally disjoint from $C$.
By Lemma~\ref{lem:KLZ} (iii), the end-vertices of $P$ are both in $Y\cap V(C)$ and $|P|\geq 5$ as $G$ is bipartite.
Let $y_1,y_2$ be the end-vertices of $P$.
If there exists a subpath $Q$ between $y_1$ and $y_2$ in $C$ of length two,
then replacing $Q$ with the path $P$ in $C$, we can get a longer cycle than $C$, a contradiction.
Fix an orientation of $C$. Then we have that $y_2\neq y_1^{++}$ (and also $y_1\neq y_2^{++}$) in $C$.
If $x\in V(P)$, then $y\in V(xPy_i)$ for some $i\in \{1,2\}$.
Without loss of generality, suppose that $y\in V(xPy_1)$.
Replacing $y_1y_1^+y_1^{++}$ with the path $y_1Px\cup xy_1^{++}$ in $C$,
again we have a longer cycle than $C$, a contradiction.
So $x\notin V(P)$. Let $C'$ be the cycle obtained from $C$ by deleting
the edges in $y_1y_1^+y_1^{++}\cup y_2y_2^+y_2^{++}$ and adding the paths $P$ and $y_1^{++}xy_2^{++}$.
Then $C'$ is a longer cycle than $C$, a contradiction. This proves the claim.

Note that every vertex in $X$ has at least $d-c$ neighbors outside $C$.
By the previous claim, these neighbors all have degree one in $G$ and
thus are distinct for different vertices in $X$.
This shows that $|Y|\geq c+|X|(d-c)\geq c+(c+1)(d-c)=c(d-c)+d$.
Since $2\leq c\leq |X|-1\leq d-2$,
now we can infer that $|Y|\geq 3d-4$, a contradiction.
This completes the proof of the lemma.
\end{proof}

We remark that the condition $|X|\leq d-1$ in Lemma~\ref{Lem:Refine-KLZ}
cannot be relaxed to $|X|\leq d$ because of the following examples.
Let $H=H(X,Y_1)$ be a complete bipartite graph with $|X|=d$ and $|Y_1|=d-1$.
Let $G$ be the bipartite graph obtained from $H$ by adding at
least one new vertex $x'$ for each vertex $x\in X$ and then adding
the edge $xx'$ for every new vertex $x'$.
Let $Y$ be the part of $G$ other than $X$.
Then the size of $Y$ can be any integer at least $2d-1$,
every vertex in $X$ has degree at least $d$ in $G$,
and $G$ is essentially 2-connected but has no cycle containing all vertices in $X$.

The following lemma will be pivotal for the proof of our main result Theorem \ref{Thm:Main}.

\begin{lemma}\label{Lem:Bipa-Path}
Let $G$ be a bipartite graph with parts $X$ and $Y$. Suppose every vertex in $X$ has degree at least $d$.
Then the followings are true:
\begin{itemize}
\item [(i)] If $|X|\leq d+1$ and $|Y|\leq 2d-1$, then $G$ has a path containing all vertices in $X$;
\item [(ii)] If $G$ is connected, $|X|\leq d$ and $|Y|\leq 3d-3$, then $G$ has a path containing all vertices in $X$;
\item [(iii)] Let $t\geq 1$ be any integer. If $|X|\leq d+t$ and $|Y|\leq 3d+2t-3$, then there exist at most $t+1$ disjoint paths in $G$ containing all vertices in $X$.
\end{itemize}
\end{lemma}

\begin{proof}
(i). It is obvious when $|X|=1$. So assume $|X|\geq 2$.
Let $G'$ be the graph obtained from $G$ by adding a new vertex $y$ and joining $y$ to every vertex in $X$.
Then every vertex in $X$ has degree at least $d+1$ in $G'$.
Since $|X|\leq d+1$ and $|Y\cup\{y\}|\leq 2d=2(d+1)-2$.
By Theorem~\ref{Thm:Jackson}, $G'$ has a cycle $C$ containing all vertices in $X$.
The vertex $y$ may be contained in $C$ or not.
In either case, one can find a path in $G$ containing all vertices in $X$ (by considering $C-y$).
This proves (i).

(ii). Similarly we may assume $|X|\geq 2$.
Let $G'$ be obtained from $G$ by adding a new vertex $y$ and joining $y$ to every vertex in $X$.
Let $Y'=Y\cup \{y\}$. Then we have that $2\leq |X|\leq d$, $|Y'|\leq 3d-2$, and every vertex in $X$ has degree at least $d+1$ in $G'$.
We claim that $G'$ is essentially 2-connected.
To see this, consider a spanning tree $T$ in $G$ (note that $G$ is connected).
Let $G''$ be obtained from $T$ by adding the vertex $y$ and joining $y$ to every vertex in $X$.
Clearly we have $G''\subseteq G'$, and by definition, $G''$ is essentially 2-connected.
This implies that $G'$ is essentially 2-connected.
Now applying Lemma \ref{Lem:Refine-KLZ}, we can conclude that $G'$ has a cycle $C$ containing all vertices in $X$.
The vertex $y$ lies on $C$ or not.
In either case, $C-y$ (and thus $G$) contains a path containing all vertices in $X$. This proves (ii).

(iii). Let $G'$ be obtained from $G$ by deleting all isolated vertices in $Y$,
adding $t$ new vertices $y_1,...,y_t$ and then joining every $y_i$ for $i\in [t]$ to all vertices in $X$.
Let $Y_0$ be the set of all isolated vertices of $G$ in $Y$ and let $Y'=(Y\backslash Y_0)\cup\{y_1,...,y_t\}$.
So $G'$ is a connected bipartite graph with parts $X$ and $Y'$ such that $|X|\leq d+t$, $|Y'|\leq |Y|+t\leq 3(d+t)-3$,
and every vertex in $X$ has degree at least $d+t$ in $G'$.
By (ii), $G'$ has a path $P$ containing all vertices in $X$.
By deleting the vertices $y_1,...,y_t$,
we can obtain at most $t+1$ disjoint paths (in $G$) from $P$ containing all vertices of $X$.
This proves Lemma~\ref{Lem:Bipa-Path}.
\end{proof}

Lastly, we need the following useful lemma due to Erd\H{o}s, Faudree, Schelp
and Simonovits \cite{EFSS89}. An analog for cycles can be found in \cite{LRWZ12}.

\begin{lemma}[\rm {Erd\H{o}s et al., \cite[Lemma~1]{EFSS89}}]\label{Thm:EFSS-Path}
Let $G$ be a graph with at most $2d+1$ vertices and $\mathcal{P}$ be any
family of disjoint paths $P_i$, where both end-vertices of each
$P_i\in \mathcal{P}$ have degree at least $d$ in $G$.
Then $G$ contains a path $Q$ such that both its end-vertices have
degree at least $d$ in $G$ and $\bigcup_{P_i\in\mathcal{P}}V(P_i)\subseteq V(Q)$.
\end{lemma}

\section{Proof of Theorem~\ref{Thm:Main}}\label{Sec:Thm:Pr-Main}
For given positive integers $n>d\geq k$, let $\phi$ be the function
$\phi(n,d,k)$ defined in Theorem~\ref{Thm:Main} throughout this section.
To complete the proof of Theorem~\ref{Thm:Main}, in view of Lemma~\ref{Lem:Ext-graphs},
it suffices to prove that
\begin{equation}\label{equ:main}
\mbox{any $n$-vertex graph with at least $\phi$ vertices of degree at least $d$ contains a path $P_{k+1}$.}
\end{equation}

We will prove this by contradiction.
Consider any positive integers $d\geq k$ for which \eqref{equ:main} fails.
Then there exists a counterexample $G$ to the statement \eqref{equ:main} as follows:
\begin{itemize}
\item [(a).] $G$ is an $n$-vertex graph with at least $\phi=\phi(n,d,k)$ vertices of degree at least $d$,
\item [(b).] $G$ does not contain any path $P_{k+1}$ (on $k+1$ vertices),
\item [(c).] subject to (a) and (b), $n$ is minimum, and
\item [(d).] subject to (a), (b) and (c), $G$ has the minimum number of edges.
\end{itemize}

We proceed the proof by proving a sequence of claims.
In the rest of the proof, we say a vertex is a {\it high degree vertex} if it has
degree at least $d$ in $G$ and {\it low degree vertex} otherwise. A path is
a {\it high-end} path if both its end-vertices are high degree vertices.
Our first claim is the following.

\begin{claim}\label{Cl:HighEndPath}
$G$ has no high-end path on at least $k-1$ vertices.
\end{claim}

\begin{proof}
Suppose that $P$ is a high-end path on at least $k-1$ vertices.
Let $u,v$ be end-vertices of $P$. If $|P|\geq k+1$, then there is nothing to prove.
If $|P|=k$, as $d(u)\geq d\geq k$, $u$ has a neighbor $u'$ outside $V(P)$, and $P\cup uu'$ is a desired path $P_{k+1}$.
So $|P|=k-1$. Then $u$ has a neighbor $u'$ outside $V(P)$, and $v$ has a neighbor $v'$ outside $V(P)\cup\{u'\}$.
Now $u'u\cup P\cup vv'$ is a path $P_{k+1}$. In any case, we get a contradiction.
\end{proof}

\begin{claim}\label{Cl:Smallk}
We may assume that $k\geq 5$.
\end{claim}

\begin{proof}
We first point out that the cases $k\in \{1,2\}$ are trivial:
a graph $G$ contains $P_2$ if and only if $G$ has a vertex of degree at least 1,
and $G$ contains $P_3$ if and only if $G$ has a vertex of degree at least 2.

Now consider the case $k=3$. So $G$ contains no $P_4$.
Then each component $H$ of $G$ is a $K_1,K_2,K_3$ or a star $K_{1,s}$ for some $s\geq 2$.
Note that only the component $K_{1,s}$ with $s\geq d$ can have one high degree vertex.
It follows that $H$ has at most $\lfloor\frac{|V(H)|}{d+1}\rfloor$ high degree vertices. So $G$ has at most
$$\sum_{\mbox{ each component } H} \left\lfloor\frac{|V(H)|}{d+1}\right\rfloor \leq\left\lfloor\frac{n}{d+1}\right\rfloor=\phi-1$$ high
degree vertices, a contradiction.

Finally consider the case $k=4$. In this case, $G$ contains no $P_5$.
Let $H$ be any component of $G$.
By Claim~\ref{Cl:HighEndPath}, $H$ has no high-end path on 3 vertices.
This shows that $H$ has at most two high degree vertices, and if $u,v$
are the two high degree vertices of $H$, then $uv$ must be a cut-edge
of $H$ and thus $N(u)\cap N(v)=\emptyset$. It also follows that if
$H$ has only one high degree vertex, then $|V(H)|\geq d+1$ and if
$H$ has two high degree vertices, then $|V(H)|\geq 2d$. So we can conclude
that $H$ has at most $\frac{|V(H)|}{d}$ high degree vertices (and if $H$ has
exactly $\frac{|V(H)|}{d}$ high degree vertices, then $|V(H)|=2d$).
Therefore $G$ has at most
$$\sum_{\mbox{ each component } H} \frac{|V(H)|}{d}\leq\frac{n}{d}$$ high
degree vertices. (We also see that in this case, if $G$ has exactly
$\frac{n}{d}$ high degree vertices, then every component of $G$ forms
a double-star $H^*_{d,4}$.) This gives that the number of high degree
vertices of $G$ is at most $\phi-1$, a contradiction.
\end{proof}

We now discuss several useful properties that the graph $G$ has.

\begin{claim}\label{Cl:NeighborLower}
If $u$ is a low degree vertex in $G$, then every neighbor of $u$ has degree exactly $d$ and thus is high.
\end{claim}

\begin{proof}
Suppose, otherwise, that there is a vertex $v\in N(u)$ with $d(v)\leq d-1$ or $d(v)\geq d+1$.
Then $G'=G-uv$ is a graph satisfying $(a), (b)$ and $(c)$, but having less edges than $G$.
This violates $(d)$ and the choice of $G$.
\end{proof}

\begin{claim}\label{Cl:Connected}
$G$ is connected.
\end{claim}

\begin{proof}
Suppose that $G$ is not connected. Then $G$ is a disjoint union of two subgraphs $G_1$ and $G_2$.
By our choice of $G$, each of $G_1$ and $G_2$ does not violate the statement \eqref{equ:main}.
Note that each $G_i$ contains no $P_{k+1}$.
For $i\in \{1,2\}$, let $|V(G_i)|:=n_i=q_i(d+1)+r_i$ for $0\leq r_i\leq d$, and let $n=q(d+1)+r$ for $0\leq r\leq d$.
As $n=n_1+n_2$, we have that either (1) $q=q_1+q_2$ and $r=r_1+r_2\leq d$, or (2) $q=q_1+q_2+1$ and $0\leq r=r_1+r_2-(d+1)\leq d$.

First consider that $k$ is odd, or $k\geq 6$ is even and $r_i\leq d-\frac{k}{2}$ for both $i\in \{1,2\}$.
In this case, since $G_i$ has no $P_{k+1}$ and does not violate the statement \eqref{equ:main},
each $G_i$ has at most $q_i\lfloor\frac{k-1}{2}\rfloor$ high degree vertices.
So $G$ has at most $(q_1+q_2)\lfloor\frac{k-1}{2}\rfloor\leq q\lfloor\frac{k-1}{2}\rfloor\leq \phi-1$
high degree vertices, a contradiction.

Now assume that $k\geq 6$ is even and exactly one of $r_1$ and $r_2$ is at most $d-\frac{k}{2}$.
Without loss of generality, assume that $r_1\leq d-\frac{k}{2}$ and $r_2>d-\frac{k}{2}$.
Then $G_1$ has at most $q_1\frac{k-2}{2}$ high degree vertices and $G_2$ has at most
$q_2\frac{k-2}{2}+1$ high degree vertices. Thus $G$ has at most $(q_1+q_2)\frac{k-2}{2}+1$
high degree vertices. Note that either $q=q_1+q_2+1$, or $q=q_1+q_2$ and $d\geq r=r_1+r_2>d-\frac{k}{2}$.
In both cases, we see that the number of high degree vertices in $G$ is at most $(q_1+q_2)\frac{k-2}{2}+1\leq \phi-1$,
a contradiction.

Finally, $k\geq 6$ is even and $r_i>d-\frac{k}{2}$ for both $i\in \{1,2\}$.
In this case, each $G_i$ has at most $q_i\frac{k-2}{2}+1$ high degree vertices,
and so $G$ has at most $(q_1+q_2)\frac{k-2}{2}+2$ high degree vertices.
Since $r_1+r_2>2d-k\geq d$, we must have that $q=q_1+q_2+1$.
Thus the number of high degree vertices in $G$ is at most
$(q_1+q_2)\frac{k-2}{2}+2\leq q\frac{k-2}{2}\leq \phi-1$, where the first inequality holds because $k\geq 6$.
This contradiction completes the proof of Claim~\ref{Cl:Connected}.
\end{proof}

\begin{claim}\label{Cl:ExactlyPhiHigh}
$G$ has exactly $\phi$ high degree vertices.
\end{claim}

\begin{proof}
Suppose that $G$ has at least $\phi+1$ high degree vertices.
If every vertex is high, then $\delta(G)\geq d\geq k$ and by Erd\H{o}s-Gallai Theorem \cite{EG59},
there is a path of length $d\geq k$ in $G$, a contradiction.
So there is some low degree vertex, say $u$ in $G$.
By Claim \ref{Cl:Connected}, $G$ is connected, so there exists some $v\in N(u)$.
By Claim \ref{Cl:NeighborLower}, $v$ is high.
Then $G'=G-uv$ is a graph satisfying $(a), (b)$ and $(c)$, but
with less number of edges, contradicting the choice of $G$.
This proves the claim.
\end{proof}

In the rest of the proof, we write $n=q(d+1)+r$ for some integers $q\geq 1$ and $0\leq r\leq d$.
Since $k\geq 5$, we have $\phi=\lfloor \frac{k-1}{2}\rfloor q+1+\epsilon$,
where $\epsilon=1$ if $k$ is even and $r>d-\frac{k}{2}$, and $\epsilon =0$ otherwise.

\begin{claim}\label{Cl:n>2d+1}
We have $n\geq 2d+2$ and thus $q\geq 2$.
\end{claim}
\begin{proof}
Suppose that $d<n\leq 2d+1$. So $q=\lfloor\frac{n}{d+1}\rfloor=1$ and $n=d+1+r$.
We will reach a contradiction to Claim~\ref{Cl:HighEndPath} by finding a
high-end path in $G$ on at least $k-1$ vertices.
Let $X$ be the set of all high degree vertices in $G$ and $Y=V(G)\backslash X$.
Let $G'=G(X,Y)$ be the spanning bipartite subgraph of $G$ with parts $X$ and $Y$.
We have $|X|=\phi$ by Claim~\ref{Cl:ExactlyPhiHigh}.

Suppose that $k$ is odd. Then $|X|=\phi=\frac{k+1}{2}$, and $|Y|=n-|X|\leq 2d+1-\frac{k+1}{2}$.
For every $x\in X$, $d_{G'}(x)\geq d_G(x)-(|X|-1)\geq d-\frac{k-1}{2}:=d_1$.
Since $d\geq k$, we can derive that $|X|=\frac{k+1}{2}\leq d_1$ and $|Y|\leq 2d+1-\frac{k+1}{2}\leq 3d_1-1$.
By Lemma~\ref{Lem:Bipa-Path} (iii) with $t=1$, $G'$ has at most two disjoint paths
(say $P_1, P_2$) containing all vertices in $X$. We may assume that all
end-vertices of $P_1, P_2$ are in $X$, so $P_1, P_2$ are high-end paths of $G$.
As $n\leq 2d+1$, by Lemma~\ref{Thm:EFSS-Path}, there is a high-end path $P$
in $G$ satisfying $V(P_1)\cup V(P_2)\subseteq V(P)$ and thus $|P|\geq |P_1|+|P_2|=2|X|-2=k-1$.

Now suppose that $k$ is even and $r\leq d-\frac{k}{2}$.
Then $|X|=\phi=\frac{k}{2}$ and $|Y|=n-|X|=(d+1+r)-|X|\leq d+1+(d-\frac{k}{2})-\frac{k}{2}=2d+1-k$.
Also for every $x\in X$, $d_{G'}(x)\geq d-\frac{k}{2}+1:=d_2$.
Since $d\geq k$, we see $|X|\leq d_2$ and $|Y|\leq 2d_2-1$.
By Lemma~\ref{Lem:Bipa-Path} (i), $G'$ has a path $P$ containing all vertices in $X$.
We may view $P$ as a path with both end-vertices in $X$.
So $P$ is a high-end path of $G$ with $|P|=2|X|-1=k-1$.

It remains to consider the case when $k$ is even and $d\geq r>d-\frac{k}{2}$.
In this case, $|X|=\phi=\frac{k}{2}+1$, $|Y|=n-|X|=(d+1+r)-|X|\leq 2d-\frac{k}{2}$,
and for every $x\in X$, $d_{G'}(x)\geq d-\frac{k}{2}:=d_3$.
Since $d\geq k$, one can deduce that $|X|\leq d_3+1$ and $|Y|\leq 3d_3$.
By Lemma~\ref{Lem:Bipa-Path} (iii) with $t=2$,
$G'$ has at most three disjoint paths $P_1,P_2,P_3$ containing all vertices in $X$,
all of which can be viewed as high-end paths of $G$.
Using Lemma~\ref{Thm:EFSS-Path},
$G$ has a high-end path $P$ containing all vertices in $P_1\cup P_2\cup P_3$.
Therefore $|P|\geq |P_1|+|P_2|+|P_3|=2|X|-3=2\left(\frac{k}{2}+1\right)-3=k-1$.

In any case, we get a contradiction to Claim~\ref{Cl:HighEndPath}.
This finishes the proof of Claim~\ref{Cl:n>2d+1}.
\end{proof}

\begin{claim}\label{Cl:NSHighVertices}
If $T$ is any set of at least $d+1$ vertices, then $T\cup N(T)$ contains
at least $\lfloor\frac{k+1}{2}\rfloor$ high degree vertices.
\end{claim}

\begin{proof}
Suppose that $T\cup N(T)$ contains at most $\lfloor\frac{k-1}{2}\rfloor$ high
degree vertices. Let $T'$ be any subset of $T$ with $|T'|=d+1$, and let
$G'$ be obtained from $G$ by deleting all vertices in $T'$. Then
$n':=|V(G')|=n-(d+1)>d$ and $G'$ has at least
$\phi-\left\lfloor\frac{k-1}{2}\right\rfloor=\phi(n',d,k)$ high degree vertices.
Thus $G'$ is a counterexample smaller than $G$ (which satisfies
$(a)$ and $(b)$ but violates $(c)$), a contradiction.
\end{proof}

Since $G$ is connected and has $\phi\geq 2$ high degree vertices,
there exist high-end paths in $G$.
Now we choose a high-end path $P$ in $G$ such that
the number of high degree vertices in $P$ is maximum, and subject to this, $|P|$ is maximum.
By Claim \ref{Cl:HighEndPath}, we have $|P|\leq k-2$.
Note that as $q\geq 2$, we have
$\phi\geq \lfloor \frac{k-1}{2}\rfloor q+1\geq k-1$.
So there must be some high degree vertex outside $V(P)$.

Let $u_1,u_2$ be the two end-vertices of $P$. We assign the
orientation of $P$ from $u_1$ to $u_2$.

\begin{claim}\label{Cl:PathDegreeS}
Let $S_i=N_{G-P}(u_i)$ for $i\in \{1,2\}$ and $S=S_1\cup S_2$.
Then any vertex in $S$ is a low degree vertex,
$S_1\cap S_2=\emptyset$, and $N_P(u_1)\cup \{u_1\}\subseteq V(P)\backslash (N_P(u_2))^+$.
In particular, we have $d_P(u_1)+d_P(u_2)\leq|P|-1$.
\end{claim}

\begin{proof}
First any vertex in $S$ is a low degree vertex (as, otherwise,
say $s\in N_{G-P}(u_1)$ is high, then $P\cup u_1s$ contradicts the choice of $P$).
If $u_1,u_2$ has a common neighbor say $v$ outside $V(P)$,
then $v$ is a low degree vertex and let $C=P\cup u_2vu_1$;
if there exists a vertex $v\in N_P(u_1)\cap (N_P(u_2))^+$ such that $u_1v,u_2v^-\in E(G)$,
then let $C=(P-v^-v)\cup \{u_1v, u_2v^-\}$. In any case $C$ is a
cycle containing all vertices in $P$. As we just discussed before
this claim, there exists some high degree vertex $x$ outside $V(P)$ (and thus outside $V(C)$).
Now let $P'$ be any path in $G$ from $x$ to $C$ and internally disjoint
from $C$ (recall Claim~\ref{Cl:Connected} that $G$ is connected).
Then $C\cup P'$ contains a high-end path containing more high degree vertices
than $P$, a contradiction to the choice of $P$.
\end{proof}

In the following claims, we investigate more properties on the sets $S_1$ and $S_2$.

\begin{claim}\label{Cl:|S|}
Every vertex in $N(S_1)\cup N(S_2)$ is high and lies in $P$,
every vertex in $(N(S_1)\backslash\{u_1\})^-\cup (N(S_2)\backslash\{u_2\})^+$ is low,
and moreover, $|S|\geq d+3$.
\end{claim}

\begin{proof}
By Claim~\ref{Cl:NeighborLower}, any neighbor of a low degree vertex
is a high degree vertex. So for $i\in \{1,2\}$, every vertex in
$N(S_i)$ is high and must lie on $P$ (otherwise, $P$ can be extended to
a longer high-end path, which contradicts the choice of $P$).
Consider any $v\in S_1$ and $x\in N(v)\backslash\{u_1\}$. So $x\in V(P)\backslash\{u_1\}$.
If $x^-$ is high, then $P'=x^-Pu_1\cup u_1vx\cup xPu_2$ is a high-end
path such that $V(P)\subsetneq V(P')$, a contradiction to the choice of $P$.
Thus we conclude that all vertices in $(N(S_1)\backslash\{u_1\})^-$ are low, and similarly
all vertices in $(N(S_2)\backslash\{u_2\})^+$ are low.
By Claim~\ref{Cl:PathDegreeS}, one can get that
$|S|=|S_1|+|S_2|\geq (d(u_1)-d_P(u_1))+(d(u_2)-d_P(u_2))\geq 2d-(k-3)\geq d+3$.
\end{proof}

\begin{claim}\label{Cl:Consecutive}
There is at most one vertex $v\in V(P)$ such that $v,v^+\in N(S)$.
In particular, we have $v\in N(S_1)\backslash N(S_2)$ and $v^+\in N(S_2)\backslash N(S_1)$.
\end{claim}

\begin{proof}
First, both $v,v^+\in N(S)$ are high.
If $v\in N(S_2)$, then by Claim~\ref{Cl:|S|}, $v^+$ is low, a contradiction.
Thus $v\in N(S_1)\backslash N(S_2)$, and similarly, $v^+\in N(S_2)\backslash N(S_1)$.
Note that possibly $v=u_1$ or $v^+=u_2$.

Suppose for a contradiction that there are two such vertices, say $v_1$ and $v_2$.
Assume that $v_1\in V(u_1Pv_2)$.
Recall that $v_1,v_2\in N(S_1)$ and $v_1^+,v_2^+\in N(S_2)$.
We remark that $v_1^+\neq v_2$ (because $v_2^-$ is low by Claim~\ref{Cl:|S|}).
So $v_1, v_1^+, v_2, v_2^+$ appear on $P$ in order.
For $i\in \{1,2\}$, let $x_i\in S_1$ with $x_iv_i\in E(G)$ and $y_i\in S_2$ with $y_iv_i^+\in E(G)$ such that
if $v_1=u_1$ then choose $x_1=x_2$, and if $v_2^+=u_2$ then choose $y_1=y_2$.
Now we define a new path $P'$ as follows (see Figure \ref{figure2}):
$$P'=\left\{\begin{array}{ll}
  u_1Pv_1\cup v_1x_1v_2\cup v_2Pv_1^+\cup v_1^+y_1v_2^+\cup v_2^+Pu_2,  & \mbox{if }x_1=x_2\mbox{ and  }y_1=y_2;\\
  u_1^+Pv_1\cup v_1x_1u_1x_2v_2\cup v_2Pv_1^+\cup v_1^+y_1v_2^+\cup v_2^+Pu_2,  & \mbox{if }x_1\neq x_2\mbox{ and  }y_1=y_2;\\
  u_1Pv_1\cup v_1x_1v_2\cup v_2Pv_1^+\cup v_1^+y_1u_2y_2v_2^+\cup v_2^+Pu_2^-,  & \mbox{if }x_1=x_2\mbox{ and  }y_1\neq y_2;\\
  u_1^+Pv_1\cup v_1x_1u_1x_2v_2\cup v_2Pv_1^+\cup v_1^+y_1u_2y_2v_2^+\cup v_2^+Pu_2^-,  & \mbox{if }x_1\neq x_2\mbox{ and  }y_1\neq
  y_2.
\end{array}\right.$$
Then $V(P')=V(P)\cup \{x_1, x_2, y_1, y_2\}$, and the possible end-vertices $u_1^+$ and $u_2^-$ of $P'$ can be low degree vertices.
Let $P''$ be the path obtained from $P'$ by removing its low end-vertices (which only can be $u_1^+$ and/or $u_2^-$).
Note that $u_1^+$ (respectively, $u_2^-$) is an end-vertex of $P'$ if and only if $x_1\neq x_2$ (respectively, $y_1\neq y_2$).
By Claim~\ref{Cl:NeighborLower}, $P''$ is a high-end path and contains exactly the same high degree vertices as in $P$.
However, $P''$ is longer than $P$, as $|P''|\geq |P'|-\mathbf{1}_{x_1\neq x_2}-\mathbf{1}_{y_1\neq y_2}\geq |P|+2$,
where the indicator function $\mathbf{1}_{x_1\neq x_2}$ equals $1$ if $x_1\neq x_2$ and $0$ otherwise
(the definition of $\mathbf{1}_{y_1\neq y_2}$ is similar).
This contradicts the choice of $P$.
\end{proof}

\begin{figure}[h]
\begin{center}
\begin{picture}(290,120)
\put(0,60){\put(20,60){\line(1,0){100}} \put(20,60){\circle*{4}}
\put(30,60){\circle*{4}} \put(50,60){\circle*{4}}
\put(60,60){\circle*{4}} \put(80,60){\circle*{4}}
\put(90,60){\circle*{4}} \put(110,60){\circle*{4}}
\put(120,60){\circle*{4}} \put(30,30){\circle*{4}}
\put(110,30){\circle*{4}} \put(30,30){\line(-1,3){10}}
\put(110,30){\line(1,3){10}} \thicklines {\color{blue}
\put(30,30){\line(2,3){20}} \put(30,30){\line(5,3){50}}
\put(110,30){\line(-5,3){50}} \put(110,30){\line(-2,3){20}}
\put(20,60){\line(1,0){30}} \put(60,60){\line(1,0){20}}
\put(90,60){\line(1,0){30}}} \put(17,65){$u_1$} \put(29,65){$u_1^+$}
\put(45,65){$v_1$} \put(57,65){$v_1^+$} \put(73,65){$v_2$}
\put(85,65){$v_2^+$} \put(101,65){$u_2^-$} \put(115,65){$u_2$}
\put(27,22){$x_1=x_2$} \put(82,22){$y_1=y_2$}
}

\put(140,60){\put(20,60){\line(1,0){100}} \put(20,60){\circle*{4}}
\put(30,60){\circle*{4}} \put(50,60){\circle*{4}}
\put(60,60){\circle*{4}} \put(80,60){\circle*{4}}
\put(90,60){\circle*{4}} \put(110,60){\circle*{4}}
\put(120,60){\circle*{4}} \put(20,30){\circle*{4}}
\put(40,30){\circle*{4}} \put(110,30){\circle*{4}}
\put(110,30){\line(1,3){10}} \thicklines {\color{blue}
\put(20,30){\line(0,1){30}} \put(20,30){\line(1,1){30}}
\put(40,30){\line(-2,3){20}} \put(40,30){\line(4,3){40}}
\put(110,30){\line(-5,3){50}} \put(110,30){\line(-2,3){20}}
\put(30,60){\line(1,0){20}} \put(60,60){\line(1,0){20}}
\put(90,60){\line(1,0){30}}} \put(17,65){$u_1$} \put(29,65){$u_1^+$}
\put(45,65){$v_1$} \put(57,65){$v_1^+$} \put(73,65){$v_2$}
\put(85,65){$v_2^+$} \put(101,65){$u_2^-$} \put(115,65){$u_2$}
\put(17,22){$x_1$} \put(37,22){$x_2$} \put(82,22){$y_1=y_2$}
}

\put(0,-15){\put(20,60){\line(1,0){100}} \put(20,60){\circle*{4}}
\put(30,60){\circle*{4}} \put(50,60){\circle*{4}}
\put(60,60){\circle*{4}} \put(80,60){\circle*{4}}
\put(90,60){\circle*{4}} \put(110,60){\circle*{4}}
\put(120,60){\circle*{4}} \put(30,30){\circle*{4}}
\put(100,30){\circle*{4}} \put(120,30){\circle*{4}}
\put(30,30){\line(-1,3){10}} \thicklines {\color{blue}
\put(30,30){\line(2,3){20}} \put(30,30){\line(5,3){50}}
\put(100,30){\line(-4,3){40}} \put(100,30){\line(2,3){20}}
\put(120,30){\line(-1,1){30}} \put(120,30){\line(0,1){30}}
\put(20,60){\line(1,0){30}} \put(60,60){\line(1,0){20}}
\put(90,60){\line(1,0){20}}} \put(17,65){$u_1$} \put(29,65){$u_1^+$}
\put(45,65){$v_1$} \put(57,65){$v_1^+$} \put(73,65){$v_2$}
\put(85,65){$v_2^+$} \put(101,65){$u_2^-$} \put(115,65){$u_2$}
\put(27,22){$x_1=x_2$} \put(97,22){$y_1$} \put(117,22){$y_2$}
}

\put(140,-15){\put(20,60){\line(1,0){100}} \put(20,60){\circle*{4}}
\put(30,60){\circle*{4}} \put(50,60){\circle*{4}}
\put(60,60){\circle*{4}} \put(80,60){\circle*{4}}
\put(90,60){\circle*{4}} \put(110,60){\circle*{4}}
\put(120,60){\circle*{4}} \put(20,30){\circle*{4}}
\put(40,30){\circle*{4}} \put(100,30){\circle*{4}}
\put(120,30){\circle*{4}} \thicklines {\color{blue}
\put(20,30){\line(0,1){30}} \put(20,30){\line(1,1){30}}
\put(40,30){\line(-2,3){20}} \put(40,30){\line(4,3){40}}
\put(100,30){\line(-4,3){40}} \put(100,30){\line(2,3){20}}
\put(120,30){\line(-1,1){30}} \put(120,30){\line(0,1){30}}
\put(30,60){\line(1,0){20}} \put(60,60){\line(1,0){20}}
\put(90,60){\line(1,0){20}}} \put(17,65){$u_1$} \put(29,65){$u_1^+$}
\put(45,65){$v_1$} \put(57,65){$v_1^+$} \put(73,65){$v_2$}
\put(85,65){$v_2^+$} \put(101,65){$u_2^-$} \put(115,65){$u_2$}
\put(17,22){$x_1$} \put(37,22){$x_2$} \put(97,22){$y_1$}
\put(117,22){$y_2$}
}
\end{picture}\label{figure2}
\caption{Key steps in the proof of Claim~\ref{Cl:Consecutive}.}
\end{center}
\end{figure}
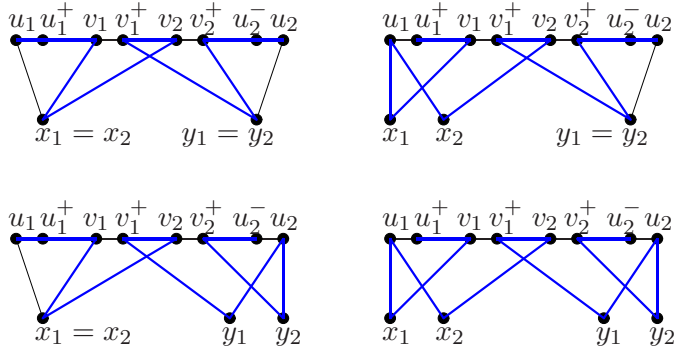

Now we can derive from Claims~\ref{Cl:|S|} and \ref{Cl:Consecutive} that
$|N(S)|\leq\frac{|P|+2}{2}\leq\frac{k}{2}.$
On the other hand, as $|S|\geq d+3$ (by Claim~\ref{Cl:|S|}),
$S\cup N(S)$ contains at least $\lfloor\frac{k+1}{2}\rfloor$ high degree vertices (by Claim~\ref{Cl:NSHighVertices}).
All vertices in $S$ are low (by Claim~\ref{Cl:PathDegreeS}), so $|N(S)|\geq \lfloor\frac{k+1}{2}\rfloor$.
Combining the above inequalities, we have
\begin{equation}\label{equ:N(S)}
\left\lfloor\frac{k+1}{2}\right\rfloor\leq |N(S)|\leq\frac{|P|+2}{2}\leq\frac{k}{2}.
\end{equation}
This indicates that $k$ is even, $|P|=k-2$, $|N(S)|=\frac{k}{2}$, and there is
exact one vertex $v\in V(P)$ satisfying Claim~\ref{Cl:Consecutive}.
Furthermore, by letting $k=2s$, we can express
$$P=a_0b_1a_1\cdots b_ja_ja_{j+1}b_{j+1}\cdots b_{s-1}a_{s-1}$$
for some $0\leq j\leq s-2$ such that $a_0=u_1,a_j=v,a_{j+1}=v^+, a_{s-1}=u_2$ and $N(S)=\{a_0, a_1,...,a_{s-1}\}$.

\begin{claim}\label{Cl:NeighborS}
We have $N(S_1)=\{a_0,...,a_j\}$ and $N(S_2)=\{a_{j+1},...,a_{s-1}\}$.
\end{claim}

\begin{proof}
It suffices to show that $N(S_1)\subseteq \{a_0,...,a_j\}$ and $N(S_2)\subseteq \{a_{j+1},...,a_{s-1}\}$.
By symmetry, we will only prove that $N(S_1)\subseteq \{a_0,...,a_j\}$.
Suppose not. Then there exists some $a_\ell\in N(S_1)$ with $\ell\geq j+1$,
and we may assume that subject to the condition $\ell\geq j+1$, $\ell$ is minimal.
By Claim~\ref{Cl:|S|}, we see that $(a_\ell)^-$ is a low degree vertex.
This implies that $\ell\geq j+2$ and $b_{\ell-1}=(a_\ell)^-$ is low.
By Claim \ref{Cl:|S|}, $a_{\ell-1}\notin N(S_1)$.
As $N(S)=\{a_0, a_1,...,a_{s-1}\}$, we have $a_{\ell-1}\in N(S_2)$.
Let $z_1\in S_1, z_2\in S_2$ be two vertices such that $z_1a_\ell, z_2a_{\ell-1}\in E(G)$.
Then $P':=a_jPu_1\cup u_1z_1a_\ell\cup a_\ell Pu_2\cup u_2z_2a_{\ell-1}\cup a_{\ell-1}Pa_{j+1}$ is a high-end path
such that $V(P')=(V(P)\backslash \{b_{\ell-1}\})\cup \{z_1,z_2\}$ and $|P'|=|P|+1$.
So $P'$ contains all high degree vertices of $P$ and is longer than $P$, a contradiction to the choice of $P$.
This proves Claim~\ref{Cl:NeighborS}.
\end{proof}

Hence by Claim~\ref{Cl:|S|}, every $a_i$ for $0\leq i\leq s-1$ is high and every $b_i$ for $1\leq i\leq s-2$ is low.

\begin{claim}\label{Cl:VertexinP}
For any $a_i\notin \{v,v^+\}$, every vertex $z\in N_{G-P}(a_i)$ is a
low degree vertex such that $N(z)\subseteq V(P)$.
On the other hand, every $b_i$ satisfies $N(b_i)\subseteq V(P)$.
\end{claim}

\begin{proof}
First we consider $b_i$ for any $1\leq i\leq s-2$.
Without loss of generality, we may assume $1\leq i\leq j$.
Suppose that $b_i$ has a neighbor $z$ outside $V(P)$.
Then by Claim~\ref{Cl:NeighborLower}, $z$ is high.
Let $w\in S_1$ be such that $wa_i\in E(G)$.
Now $zb_i\cup b_iPu_1\cup u_1wa_i\cup a_iPu_2$ is a high-end
path containing more high degree vertices than $P$, a contradiction.
This proves $N(b_i)\subseteq V(P)$.
Now consider any $a_i\notin \{v,v^+\}$.
We are done by Claims~\ref{Cl:PathDegreeS} and \ref{Cl:|S|} if $a_i\in\{u_1,u_2\}$.
Hence, without loss of generality assume that $1\leq i\leq j-1$.
Consider $z\in N_{G-P}(a_i)$.
Let $w'\in S_1$ be such that $w'a_{i+1}\in E(G)$
and let $P':=za_i\cup a_iPu_1\cup u_1w'a_{i+1}\cup a_{i+1}Pu_2$.
If $z$ is high, then $P'$ is a high-end path containing more
high degree vertices than $P$, a contradiction.
So any such $z$ must be low. Suppose for a contradiction
that $z$ has a neighbor $z'$ outside $V(P)$.
Then by Claim~\ref{Cl:NeighborLower}, $z'$ is high.
So $z'z\cup P'$ is a high-end path containing more high
degree vertices than $P$, again a contradiction.
Thus we can conclude that $N(z)\subseteq V(P)$.
\end{proof}

We are ready to complete the proof of Theorem~\ref{Thm:Main}.
Recall that $G$ has a high degree vertex (say $x$) outside $V(P)$.
Let $U=V(P)\backslash\{v,v^+\}$.
We can derive from Claims~\ref{Cl:|S|} and \ref{Cl:VertexinP}
that $\{v,v^+\}$ is a 2-cut of $G$ separating the vertex $x$ from the set $S\cup U$.
Let $D$ be the union of (at most two) components in $G-\{v,v^+\}$ containing vertices in $S\cup U$.
Claim~\ref{Cl:VertexinP} also shows that all high degree vertices in $D\cup N(D)$
are those in $V(P)$, i.e., vertices in $N(S)=\{a_0,a_1,...,a_{s-1}\}$.
By Claim~\ref{Cl:PathDegreeS}, we have
$$|S\cup V(P)|=|S_1|+|S_2|+|P|\geq d_{G-P}(u_1)+d_{G-P}(u_2)+ (d_P(u_1)+d_P(u_2))+1=d(u_1)+d(u_2)+1\geq 2d+1.$$
Either $v$ or $v^+$ has some neighbor not in $S\cup V(P)$.
Without loss of generality we assume that $v$ has some neighbor not in $S\cup V(P)$.
By Claim~\ref{Cl:Connected}, $v$ has degree $d$. Also note that $vv^+\in E(G)$.
So $v$ has at most $d-2$ neighbors in $S\cup U$.
Set $S'=(S\cup U)\backslash N(v).$
Then using $|S\cup V(P)|\geq 2d+1$, we have
$$|S'|\geq |S\cup U|-(d-2)=(|S\cup V(P)|-2)-(d-2)\geq d+1.$$
By Claim \ref{Cl:NSHighVertices}, $S'\cup N(S')$ contains at least $s=\frac{k}{2}$ high degree vertices.
However, as $S'$ is a subset in $D$, by definition we see $S'\cup N(S')\subseteq (D\cup N(D))\backslash \{v\}$.
We have pointed out that all high degree vertices in $D\cup N(D)$ are $a_0,a_1,...,a_{s-1}$.
So $S'\cup N(S')$ contains at most $s-1$ high degree vertices.
This final contradiction finishes the proof of Theorem~\ref{Thm:Main}. \qed

\section{On two related questions}\label{Sec:ErdosProb}
In this section, we consider two questions related to
Conjecture~\ref{Conj:EFSS} that are raised by Erd\H{o}s,
Faudree, Schelp and Simonovits in \cite{EFSS89}.
We will provide better constructions than the ones given
in \cite{EFSS89}, which give (negative and positive) answers
to their questions.

It is natural to consider the analog of Definition~\ref{Def:phi} for long cycles.
For integers $n>d\geq k\geq 2$, define $\theta(n,d,k)$ to be
the smallest integer $\theta$ such that every $n$-vertex graph
with at least $\theta$ vertices of degree at least $d$ contains a cycle on at least $k+1$ vertices.
In this language, the well-known Dirac's theorem \cite{D52} states that $\theta(n,d,d)\leq n$.
Improving Dirac's theorem, Woodall \cite{W75} proved that
$\theta(n,d,d)\leq \frac{(d+2)(n-1)}{2d}$ if $d$ is even and $\theta(n,d,d)\leq \frac{d(n-2)}{2(d-1)}$ otherwise;
while for the general case, he \cite{W75} showed that $\theta(n,d,k)\leq \frac{(k+3)(n-1)}{2d}$.
Recall the graph $H_{d,k+1}$ and observe that it contains no cycles on at least
$k+1$ vertices and has $d+1$ vertices in total, where $\lfloor\frac{k}{2}\rfloor$ vertices have degree $d$
(call them {\it high degree} vertices) and all other vertices have
degree $\lfloor\frac{k}{2}\rfloor$ (call them {\it low degree} vertices).
By considering the graphs consisting of blocks $H_{d,k+1}$,
the authors of \cite{EFSS89} raised the following question:
Is it possible that
$\theta(n,d,k)\leq \lfloor\frac{k}{2}\rfloor \lfloor\frac{n-1}{d}\rfloor+2$?

In the following proposition, we give a negative answer to this question.
\begin{prop}\label{lem:c}
For any integer $k\geq 2$,
there exist infinitely many integers $d$ such that the following holds.
There exists some constant $c=c(d,k)>0$ such that $\theta(n,d,k)>(\lfloor\frac{k}{2}\rfloor +c)\cdot \frac{n-1}{d}$ for infinite integers $n$.
\end{prop}

\begin{proof}
We show a slightly stronger statement:
Let $\alpha, \beta$ be positive integers such that $d=(1+\alpha)\lfloor\frac{k}{2}\rfloor$ and $n=1+d+\alpha \beta d$ satisfy $n>d\geq k\geq 2$.
Then $\theta(n,d,k)\geq \frac{n-1}{d} \lfloor\frac{k}{2}\rfloor +\frac{n-(d+1)}{\alpha d}$+1.

We construct an $n$-vertex graph $G$ as follows.
Let $H_0$ be a copy of $H_{d,k+1}$ with a low degree vertex $v_0$.
For each $1\leq i\leq \beta$, let $H_i$ be obtained from $\alpha$ copies of $H_{d,k+1}$ by identifying one low degree vertex from each copy of $H_{d,k+1}$ (call the resulting vertex $u_i$);
let $v_i$ be a low degree vertex in $H_i$ other than $u_i$.
Finally, let $G$ be obtained from $H_0,H_1,\ldots,H_\beta$ by
identifying $v_{i-1}$ and $u_i$ for all $1\leq i\leq \beta$.
Since each block of $G$ is a copy of $H_{d,k+1}$, we see that $G$ contains no cycles of at least $k+1$ vertices.
However, $G$ has $(1+\alpha\beta)\lfloor\frac{k}{2}\rfloor+\beta=\frac{n-1}{d} \lfloor\frac{k}{2}\rfloor+\frac{n-(d+1)}{\alpha d}$ vertices of degree at least $d$.
This proves the proposition.
\end{proof}

Note that in the above proof, one can take $\alpha\geq 1$ for even $k$ and $\alpha\geq 2$ for odd $k$.
For the cases $d>k$ in Lemma~\ref{lem:c}, the constant $c$ can be taken up to $\frac{1}{2}$ (i.e., when it corresponds to $\alpha=2$).

Another question concerned in \cite{EFSS89} is the restricted version of
Conjecture~\ref{Conj:EFSS} when the graph $G$ is assumed to be connected.
For positive integers $n>d\geq k$, define $\psi(n,d,k)$ to be the smallest integer $\psi$ such that
every $n$-vertex connected graph with at least $\psi$ vertices of
degree at least $d$ contains a path $P_{k+1}$ on $k+1$ vertices.
Erd\H{o}s et al. \cite{EFSS89} observed that
the graph, obtained from $\lfloor\frac{n-1}{d}\rfloor$ copies of
$H_{d,\lceil\frac{k}{2}\rceil+1}$ by identifying at a fixed high
degree vertex of each $H_{d,\lceil\frac{k}{2}\rceil+1}$, contains no $P_{k+1}$.
This gives that
$\psi(n,d,k)\geq \lfloor\frac{k-3}{4}\rfloor\lfloor\frac{n-1}{d}\rfloor+2,$
which is approximately a half of the number of high degree vertices in Conjecture~\ref{Conj:EFSS} (as $k\to \infty$).
They \cite{EFSS89} asked if there is a better construction.
We show that it is possible to improve the leading coefficient
of $n$ by a positive constant factor in the above lower bound of $\psi(n,d,k)$.

\begin{prop}\label{lem:c2}
For any integer $k\geq 7$, there exist infinitely many integers $d$ such that the following holds.
There exists some constant $c'=c'(d,k)>0$ such that
$\psi(n,d,k)>(\lfloor\frac{k-3}{4}\rfloor +c')\cdot\frac{n-1}{d}$
for infinite integers $n$.
\end{prop}
\begin{proof}
Let $\alpha\geq 2$ and $\beta\geq d$ be any integers.
Let $d=1+\alpha\lfloor\frac{k-3}{4}\rfloor$ and $n=1+\beta(1+\alpha d)$.
Note that the graph $H^*=H_{d,\lceil\frac{k}{2}\rceil-1}$ contains
$\lfloor\frac{k-3}{4}\rfloor$ vertices of degree $d$,
and all other vertices have degree $\lfloor\frac{k-3}{4}\rfloor$ (call them {\it low} degree vertices).
We will construct an $n$-vertex connected graph $G$ to show that
$\psi(n,d,k)\geq \lfloor\frac{k-3}{4}\rfloor \frac{n-1}{d}+\left(1-\frac{1}{d}\lfloor\frac{k-3}{4}\rfloor\right)\frac{n-1}{1+\alpha d}+2$
as follows. Let $H_0$ be a star $K_{1,\beta}$ with leaves $v_i$ for $1\leq i\leq \beta$.
For each $1\leq i\leq \beta$, let $H_i$ be obtained from $\alpha$ copies of
$H^*$ by identifying one low degree vertex from each $H^*$ (call the resulting vertex $u_i$).
Now let $G$ be obtained from $H_0, H_1, ..., H_\beta$ by identifying
$v_i$ and $u_i$ for each $1\leq i\leq \beta$.
Then $G$ contains no path $P_{k+1}$ and has
$\alpha\beta\lfloor\frac{k-3}{4}\rfloor+\beta+1=\left(\frac{n-1}{d} -\frac{\beta}{d}\right) \lfloor\frac{k-3}{4}\rfloor+ \beta+1
=\lfloor\frac{k-3}{4}\rfloor \frac{n-1}{d}+\left(1-\frac{1}{d}\lfloor\frac{k-3}{4}\rfloor\right)\frac{n-1}{1+\alpha d}+1$
vertices of degree at least $d$, as desired.
\end{proof}

It would be very interesting to determine the functions $\theta(n,d,k)$ and $\psi(n,d,k)$ exactly.

\bigskip

\noindent {\bf Acknowledgements.}
The authors are very grateful to Professor Douglas R. Woodall for sending a copy of \cite{W75}.
Binlong Li was partially supported by NSFC grant 12071370. Jie Ma was partially supported by the
National Key R and D Program of China 2020YFA0713100, NSFC grants 11622110 and 12125106, and Anhui Initiative in Quantum Information
Technologies grant AHY150200. Bo Ning was partially supported by NSFC grant 11971346.

\end{document}